\documentclass[12pt]{amsart}
\usepackage{array}
\usepackage{amstext}
\usepackage{amsfonts}
\usepackage{amsmath}
\usepackage{amssymb}
\usepackage{color}
\usepackage{enumerate}
\usepackage{filecontents}
\usepackage[colorlinks=true, linkcolor=blue]{hyperref}

\usepackage{tikz}
\usetikzlibrary{arrows, positioning}
\tikzset{Mylong/.style={text width=3.1cm, align=center}, myarr/.style={->, double equal sign distance, -implies}}

\def\cal{\mathcal}

\newtheorem{theorem}{Theorem}[section]
\newtheorem{corollary}[theorem]{Corollary}
\newtheorem{lemma}[theorem]{Lemma}

\theoremstyle{definition}
\newtheorem{definition}[theorem]{Definition}
\newtheorem{example}[theorem]{Example}
\theoremstyle{remark}
\newtheorem{remark}[theorem]{Remark}

\title[$\text{\scalebox{.9}{Distributional chaos for weighted translation operators on groups}}$]
{Distributional chaos for weighted translation operators on groups}
\author[K-Y. Chen]{Kui-Yo Chen}



\address{Department of Mathematics, National Taiwan University, Taipei 106, Taiwan}
\email{r04221001@ntu.edu.tw}
\thanks{The author was supported by grant MOST 111-2115-M-142-001- of National Science and Technology Council, Taiwan.}

\date{\today}

\begin{document}

\begin{abstract}
In this paper, we study distributional chaos for weighted translations on locally compact groups.
We give a sufficient condition for such operators to be distributionally chaotic and construct an example of distributionally chaotic weighted translations by way of the sufficient condition.
In particular, we prove the existence of distributional chaos and Li-Yorke chaos for weighted translations operators with aperiodic elements.

Furthermore, we also investigate the set of distributionally irregular vectors ($DIV$) of weighted translations through the cone and equivalence classes.
When the field is that of complex numbers, we uncover several properties on certain subsets of $DIV$, including their connectedness and correspondences with some measurable subsets in locally compact groups.
\end{abstract}

\maketitle
\addcontentsline{toc}{section}{Title}
\baselineskip17pt

\section{Introduction}
\label{sec:introduction}

In the past several decades, the study of linear dynamics has attracted a lot of attention.
At this present stage, there are some excellent books (for instance, see \cite{bmbook,gpbook}) on this topic.
Hypercyclicity and linear chaos play important roles in this investigation.
A linear operator $T$ on a separable Banach space $X$ is called {\it hypercyclic}
if there is a vector $x\in X$ such that the set $\{x, Tx, T^2x, \cdot\cdot\cdot\}$ is dense in $X$.
It is well known that hypercyclicity is equivalent to topological transitivity on separable Banach spaces.
If $T$ is hypercyclic together with the dense set of periodic points, then $T$ is said to be {\it Devaney chaotic}.

Devaney chaos is highly related to distributional chaos which was introduced by
Schweizer and Sm\'ital in \cite{ss94} and can be viewed as an extension of Li-Yorke chaos.
Recently, the notion of distributional chaos was considered for linear operators on
Banach spaces and Fr\'echet spaces in \cite{bbmp11,2013JFA,marko16,mop09,mop13,o06}.
Inspired by these, in this paper,
we initiate the investigation of distributional chaos for a class of specific linear operators,
namely, weighted translations on locally compact groups, and give some characterizations for their related properties.

The investigation of hypercyclicity and chaos on locally compact groups $G$ was first studied in \cite{cc11}.
Since then, some authors continue to tackle this theme.
Indeed, disjoint hypercyclicity of weighted translations on $L^p(G)$ was characterized in \cite{chen172}.
The existence of hypercyclic weighted translations on $L^p(G)$ was discussed in \cite{kuchen17}.
Also, Abakumov and Kuznetsova in \cite{ak17} focused on the density of translates in
the weighted Lebesgue space $L_w^p(G)$, where $w$ is a weight on $G$.
Recently, linear dynamics on Orlicz spaces $L_\Phi(G)$ and hypergroups are studied in \cite{aa17,cd18,ct18}
where $\Phi$ is a Young function.
However, distributional chaos on locally compact groups has not yet been considered in the literatures.
Hence, it is significant to study distributional chaos for weighted translations on groups.
We first recall the definition and some properties of distributional chaos for further discussions.

\subsection{Distributional chaos} 
\label{sec:distributional_chaos}

In \cite{ss94}, Schweizer and Sm\'ital introduced distributional chaos by the distributional function.
In \cite{2013JFA}, some equivalent conditions of distributional chaos were obtained for linear operators on Banach spaces.
\begin{theorem}{\rm(}\cite[Theorem 12]{2013JFA}{\rm)}
Let $T$ be a bounded linear operator on a Banach space $X$.
Then the following conditions are equivalent.
\begin{enumerate}
\item[{\rm(i)}] $T$ satisfies the Distributional Chaos Criterion (DCC);
\item[{\rm(ii)}] $T$ has a distributionally irregular vector;
\item[{\rm(iii)}] $T$ is distributionally chaotic;
\item[{\rm(iv)}] $T$ admits a distributionally chaotic pair.
\end{enumerate}
\end{theorem}

To achieve our goal,
we will use the definitions of ${\rm(i)}$ the Distributional Chaos Criterion (DCC, for short) and ${\rm(ii)}$ distributionally irregular vectors,
which will be recalled in Section \ref{sec:sufficient_conditions} and Section \ref{sec:distributionally_irregular_vectors} respectively.

Note that distributionally chaoticity implies that Li-Yorke chaoticity (see \cite{bbmp11}).

For convenience, we define the upper and lower densities here.
Let $A\subset\mathbb{N}$. The upper and lower densities of $A$ are defined by
$$\overline{\mbox{dens}}(A)=\limsup_{n\rightarrow\infty}\frac{\mbox{card}(A\cap [1,n])}{n}\qquad
\mbox{and}\qquad \underline{\mbox{dens}}(A)=\liminf_{n\rightarrow\infty}\frac{\mbox{card}(A\cap [1,n])}{n},$$
respectively.


\subsection{Weighted translation} 
\label{sec:weighted_translation}

Throughout this paper, let $G$ be a second countable locally compact Hausdorff group with a right invariant Haar measure $\lambda$,
and denote by $L^p(G)\ (1\leq p <\infty)$ the complex Lebesgue space with respect to
$\lambda$.
A bounded continuous function $w:G \rightarrow (0,\infty)$ is
called a {\it weight} on $G$. Let $a \in G$.
A {\it weighted translation operator} $T_{a,w} : L^p(G) \longrightarrow L^p(G)$ is defined by
$$T_{a,w}(f) := wT_a(f) \qquad (f \in L^p(G))$$
where $w$ is a weight on $G$ and $T_a(f)(x) := f(xa^{-1})$.

If $w^{-1}\in L^\infty(G)$, then we denote the inverse of $T_{a,w}$ by $S_{a,w}$.
So $S_{a,w}=T_{a^{-1},(T_{a^{-1}}w)^{-1}}$, which is also a weighted translation.

Given a weight $w$ on $G$, define the operator $M_w:L^p(G)\rightarrow L^p(G)$
by
$$M_w(f)(x):=w(x)f(x),\qquad (f\in L^p(G),x\in G)$$
and let
$$
\varphi_{n}:=\prod_{j=1}^{n}T_{a^{-j}}w
\ \ \ \
\mbox{and}
\ \ \ \ \
{\widetilde \varphi_{n}}:=\left(\prod_{j=0}^{n-1}T_{a^j}w\right)^{-1}.
$$

Then
one can deduce $T^n_{a,w}=M_{{\widetilde \varphi_{n}}^{-1}}T_{a^n}=T_{a^n}M_{\varphi_n}$
for all $n\in\mathbb{N}$.

In what follows,
given a subset $Y$ of $G$,
we extend each function in $L^p(Y)$ to $G$
by setting the values to zero outside of $Y$.
Thus, we can identify $L^p(Y)$ as a subspace of $L^p(G)$.

In this paper, we will focus on weighted translation operators and give a sufficient condition (we call it (DCCW), see Definition \ref{DCCW}) for
such operators to be distributionally chaotic in Section \ref{sec:sufficient_conditions} by applying (DCC) (see Theorem \ref{DCCW_theorem}).
In particular,
this characterization ensures
the existence of distributionally chaotic and Li-Yorke chaotic weighted translation $T_{a,w}$ for certain group elements $a$ (see Theorem \ref{existence}).

As a result, it becomes meaningful to describe the structures and properties of the set of distributionally irregular vectors ($DIV$) for such operators (see Section \ref{sec:distributionally_irregular_vectors}).
We describe the connectivity of some substructures $DIV_{A,B}$ of $DIV$ and
claim that the closure of $DIV$ is a union of some special closed subspaces $L^p(M)$ for some $M$s
(see Theorem \ref{DIV_A_path_connected} and Theorem \ref{m_ab}).

\section{Sufficient conditions}
\label{sec:sufficient_conditions}

We first introduce our setting to begin this section.

In \cite{2013JFA}, the authors introduce the so called Distributional Chaos Criterion.
We recall and reformulate this criterion below.

\begin{definition}{\rm(}\cite[Definition 10]{2013JFA}\label{dcc}{\rm)}
Let $T$ be a bounded linear operator on a Banach space $X$. We say that $T$ satisfies the {\it Distributional Chaos Criterion} (DCC) if
\begin{enumerate}
\item[{\rm(i)}] there exist a sequence $(x_k)$ in $X$ and $A\subset \mathbb{N}$ with $\overline{{\rm dens}}(A)=1$
such that $\displaystyle\lim_{n\in A}T^n x_k=0$ for all $k$;
\item[{\rm(ii)}] there exist $y\in \overline{{\rm span}\{x_n: n\in\mathbb{N}\}}$ and $B\subset\mathbb{N}$ with
$\overline{{\rm dens}}(B)>0$ such that $\displaystyle\lim_{n\in B}\|T^ny\|=\infty$.
\end{enumerate}
\end{definition}

\begin{remark}
We note that condition (ii) of Definition \ref{dcc} in the original version is stated as follows:
\begin{enumerate}
\item[{\rm(ii')}] $y_k\in \overline{{\rm span}\{x_n: n\in\mathbb{N}\}}$, $\|y_k\|\rightarrow0$ and there exist $\varepsilon >0$ and an increasing
sequence $(N_k)$ in $\mathbb{N}$ such that
$$\mbox{card}\{1\le j\le N_k:\|T^j y_k\|>\varepsilon\}\geq \varepsilon N_k$$
for all $k\in \mathbb{N}$.
\end{enumerate}
However, conditions (ii) and (ii') are equivalent by setting $X=\overline{{\rm span}\{x_n: n\in\mathbb{N}\}}$ in \cite[Proposition 8]{2013JFA}.
\end{remark}

In this section, we will give a sufficient condition (DCCW) for weighted translations $T_{a,w}$ to be distributionally chaotic.

\begin{definition}\label{DCCW_temp}
Let $T_{a,w}$ be a weighted translation on $L^p(G)$.
We say that $T_{a,w}$ satisfies the property $\mathfrak{P}$
if there exists a sequence of compact sets $(K_n)$ in $G$
such that
\begin{enumerate}
\item[{\rm(i)}] there exists $A\subset \mathbb{N}$ with $\overline{{\rm dens}}(A)=1$ such that
$\displaystyle\lim_{n\in A}\|\varphi_n|_{K_k}\|_p=0$ for all $k$,
\item[{\rm(ii)}] there exists $B\subset \mathbb{N}$ with $\overline{{\rm dens}}(B)>0$ such that
$$\sum_{n\in B}\frac{\lambda(K_n)^{\frac{1}{p}}}{\|\varphi_n|_{K_n}\|_p}<\infty.$$
\end{enumerate}
\end{definition}


For Definition \ref{DCCW_temp}, there is some detail discussion between $\varphi_n$ and the compact subsets $K_n$.
\begin{lemma}\label{small_Kn}
Let $T_{a,w}$ be a weighted translation on $L^p(G)$.
Suppose that $T_{a,w}$ satisfies $\mathfrak{P}$ for a sequence of compact sets $(K_n)$.
Then $T_{a,w}$ also satisfies $\mathfrak{P}$ for some compact subsets $K_n'$ of $K_n$
with $\sup\limits_{K_n'}\varphi_n^p-\inf\limits_{K_n'}\varphi_n^p$
arbitrary small.
Explicitly, for any $\varepsilon_n > 0$, there exists compact subset $K_n'\subseteq K_n$ such that
$\sup\limits_{K_n'}\varphi_n^p-\inf\limits_{K_n'}\varphi_n^p < \varepsilon_n$ and $T_{a,w}$ also satisfies $\mathfrak{P}$ with the sequence of compact subsets $K_n'$.
\end{lemma}
\begin{proof}
We first prove that the condition (ii) of $\mathfrak{P}$ holds for $K_n'$.
Choose some $\alpha \in [\frac{1}{\lambda(K_n)}\int_{K_n}\varphi_n^p, \sup\limits_{K_n}\varphi_n^p)$
such that $\sup\limits_{K_n}\varphi_n^p - \alpha < \varepsilon_n$.
Set
$$K_n':=
\left\{x \in K_n : \varphi_n^p(x) \ge \frac{\alpha + \sup\limits_{K_n}\varphi_n^p}{2}\right\}
.$$
Then we have
$$\sum\limits_{n\in B} \frac{\lambda(K_n')^{\frac{1}{p}}}{\|\varphi_n|_{K_n'}\|_p}
= \sum\limits_{n\in B}\left(\frac{1}{\frac{1}{\lambda(K_n')}\int_{K_n'}\varphi_n^p}\right)^{\frac{1}{p}}
\le \sum\limits_{n\in B}\left(\frac{1}{\frac{1}{\lambda(K_n)}\int_{K_n}\varphi_n^p}\right)^{\frac{1}{p}}
=\sum\limits_{n\in B}\frac{\lambda(K_n)^{\frac{1}{p}}}{\|\varphi_n|_{K_n}\|_p}
<\infty$$
together with $\sup\limits_{K_n'}\varphi_n^p-\inf\limits_{K_n'}\varphi_n^p < \varepsilon_n$.

Concerning the condition (i) of $\mathfrak{P}$, one always has
$$\lim\limits _{n\in A} \|T_{a,w}^n\chi_{K_k'}\|_p \le \lim\limits _{n\in A} \|T_{a,w}^n\chi_{K_k}\|_p=0$$
if $K_k'\subseteq K_k$.

Thus, $T_{a,w}$ also satisfies $\mathfrak{P}$ with the sequence of compact subsets $K_n'$.
\end{proof}

\begin{definition}\label{DCCW}
(DCCW)
Let $T_{a,w}$ be a weighted translation on $L^p(G)$. If there exists a sequence $g_n\in G$ such that
\begin{enumerate}
\item[{\rm(i)}] there exist $A\subset \mathbb{N}$ with $\overline{{\rm dens}}(A)=1$ and
a sequence of compact neighborhood $K_n$ of $g_n$
such that
$\displaystyle\lim_{n\in A}\|\varphi_n|_{K_k}\|_p=0$ for all $k$,
\item[{\rm(ii)}] there exists $B\subset \mathbb{N}$ with $\overline{{\rm dens}}(B)>0$ such that
$$\sum\limits_{n\in B} \frac{1}{\varphi_n(g_n)}<\infty,$$
\end{enumerate}
then we say that $T_{a,w}$ satisfies the {\it Distributional Chaos Criterion for Weighted Translations} (DCCW, for short).
\end{definition}

In this section, we will prove the following diagram.
\vskip1em
\hskip7.5em
\begin{tikzpicture}
\node (conI) {(DCCW)};
\node (conII) [right = of conI] {Property $\mathfrak{P}$};
\node (conIII) [right = of conII] {(DCC)};

\draw [myarr] (conI) -- (conII);
\draw [myarr] (conII) -- (conIII);
\end{tikzpicture}
\vskip1em
For the first implication, see Theorem \ref{DCCW2P}.
For the second implication, see Theorem \ref{DCCW_theorem}.



\begin{lemma}\label{equi_(ii)}
Let $T_{a,w}$ be a weighted translation on $L^p(G)$.
Suppose there exists $B\subset \mathbb{N}$ with $\overline{{\rm dens}}(B)>0$ such that
$\sum\limits_{n\in B} \frac{1}{\varphi_n(g_n)}<\infty$ for some $g_n\in G$.
Then for any given sequence of $V_n$ with $V_n$ being of neighborhood of $g_n$,
there exists a sequence of compact neighborhoods $K_n\subseteq V_n$ in $G$ such that
$$\sum_{n\in B}\frac{\lambda(K_n)^{\frac{1}{p}}}{\|\varphi_n|_{K_n}\|_p}<\infty.$$
\end{lemma}
\begin{proof}

Since $\varphi_n^p$ is continuous at $g_n$,
we can find compact neighborhoods $K_n$ small enough in $V_n$ such that
$$\Bigg|\left(\frac{1}{\varphi_n^p(g_n)}\right)^{\frac{1}{p}}
-\left(\frac{1}{\frac{1}{\lambda(K_n)}\int_{K_n}\varphi_n^p}\right)^{\frac{1}{p}}\Bigg|<\frac{1}{2^n}.$$

Thus
\[\sum\limits_{n\in B}\frac{\lambda(K_n)^{\frac{1}{p}}}{\|\varphi_n|_{K_n}\|_p}
=\sum\limits_{n\in B}\left(\frac{1}{\frac{1}{\lambda(K_n)}\int_{K_n}\varphi_n^p}\right)^{\frac{1}{p}}
\le
\sum\limits_{n\in B} \frac{1}{\varphi_n(g_n)} + 1
<\infty.\]
\end{proof}

\begin{theorem}\label{DCCW2P}
Let $T_{a,w}$ be a weighted translation on $L^p(G)$.
If 
$T_{a,w}$ satisfies (DCCW),
then
$T_{a,w}$ satisfies the property $\mathfrak{P}$.
\end{theorem}
\begin{proof}


Suppose that $T_{a,w}$ satisfies (DCCW) with $K_n$ and $g_n$.
By Lemma \ref{equi_(ii)}, there exists a sequence of compact neighborhoods $K_n'\subseteq K_n$ of $g_n$ such that
the condition (ii) in $\mathfrak{P}$ is satisfied.
Then the sequence $K_n'$ satisfies both (i) and (ii) in $\mathfrak{P}$.
\end{proof}

So far, we have proven that (DCCW) implies the property $\mathfrak{P}$.
Next, we will prove that the implication of the property $\mathfrak{P}$ to (DCC) (Definition \ref{dcc}) holds. 
Therefore, we conclude that (DCCW) is a sufficient condition for distributional chaoticity.

\begin{lemma}\label{equi_DCCW}
Let $(K_n)$ be a sequence of compact sets in $G$ with positive measures. Then the following conditions are equivalent.
\begin{enumerate}
\item[{\rm(i)}]
There exists $B\subset \mathbb{N}$ with $\overline{{\rm dens}}(B)>0$ such that
$$\sum_{n\in B}\frac{\lambda(K_n)^{\frac{1}{p}}}{\|\varphi_n|_{K_n}\|_p}<\infty.$$
\item[{\rm(ii)}]
There exist a nonnegative sequence $(c_n)$ and $B\subset \mathbb{N}$ with $\overline{{\rm dens}}(B)>0$ such that
$\displaystyle\lim_{n\in B}c_n\|\varphi_n|_{K_n}\|_p=\infty$ and $\displaystyle\sum_{n\in B}c_n\lambda(K_n)^{\frac{1}{p}}<\infty$.
\end{enumerate}
\end{lemma}
\begin{proof}
We will only prove the non-trivial implication ${\rm(i)} \Rightarrow {\rm(ii)}$.
Let
$$a_n=\left\{\begin{matrix}
& \frac{\lambda(K_n)^{\frac{1}{p}}}{\|\varphi_n|_{K_n}\|_p}
&\mbox{ if } n\in B; \\
& 0
&\mbox{ if } n\not\in B,
\end{matrix}\right.$$
$r_n=\sum\limits_{i\geq n}a_i$ and
$$c_n=\left\{\begin{matrix}
& \frac{1}{\sqrt{r_n}\|\varphi_n|_{K_n}\|_p}
&\mbox{ if } n\in B; \\
& 0
&\mbox{ if } n\not\in B. \\
\end{matrix}\right.$$
Then $\displaystyle\lim_{n\in B}c_n\|\varphi_n|_{K_n}\|_p=\infty$. Moreover,
by
$$
\frac{a_n}{\sqrt{r_n}}
=
\frac{(\sqrt{r_n}+\sqrt{r_{n+1}})(\sqrt{r_n}-\sqrt{r_{n+1}})}{\sqrt{r_n}}
\le
\frac{2\sqrt{r_n}(\sqrt{r_n}-\sqrt{r_{n+1}})}{\sqrt{r_n}}
=
2(\sqrt{r_n}-\sqrt{r_{n+1}})
,$$
we have $\sum\limits_{n\in \mathbb{N}}\frac{a_n}{\sqrt{r_n}} \leq 2\sqrt{r_1}$.
Thus, 
\begin{eqnarray*}
\sum_{n\in B}c_n\lambda(K_n)^{\frac{1}{p}}
&=&
\sum_{n\in B}\frac{a_n}{\sqrt{r_n}}
=\sum_{n\in \mathbb{N}}\frac{a_n}{\sqrt{r_n}}
\leq 2\sqrt{r_1}\\
&=&
2\sqrt{\sum_{n\in \mathbb{N}}a_n}
=2\sqrt{\sum_{n\in B}a_n}
=2\sqrt{\sum_{n\in B}\frac{\lambda(K_n)^{\frac{1}{p}}}{\|\varphi_n|_{K_n}\|_p}}
<\infty.
\end{eqnarray*}
\end{proof}

\begin{theorem}\label{DCCW_theorem}
Let $T_{a,w}$ be a weighted translation on $L^p(G)$ satisfying (DCCW).
Then it is distributionally chaotic.
\end{theorem}
\begin{proof}
By Theorem \ref{DCCW2P}, $T_{a,w}$ satisfies $\mathfrak{P}$ with compact sets $K_n$.

We first show that the condition (i) of (DCC) (Definition \ref{dcc}) holds. Define $x_k:=\chi_{K_k}$. Then
$T_{a,w}^nx_k=T_{a^n} M_{\varphi_n}\chi_{K_k} = T_{a^n}(\varphi_n\chi_{K_k})$.
Hence one obtains $\lim\limits _{n\in A} T_{a,w}^n x_k = 0$ by the fact
$\lim\limits _{n\in A} \|T_{a^n}(\varphi_n|_{K_k})\|_p = \lim\limits _{n\in A} \|\varphi_n|_{K_k}\|_p = 0$.

Next, we prove that the condition (ii) of (DCC) is satisfied as well.
By Lemma \ref{equi_DCCW}, there exist a nonnegative sequence $(c_n)$ and $B\subset \mathbb{N}$ with
$\overline{{\rm dens}}(B)>0$ such that
$\displaystyle\lim_{n\in B}c_n\|\varphi_n|_{K_n}\|_p=\infty$, and $\displaystyle\sum_{n\in B}c_n\lambda(K_n)^{\frac{1}{p}}<\infty$.
Set
$$y:=\sum_{n\in B} c_n \chi_{K_n}.$$
Then
$$\|y\|_p=\Big\|\sum\limits_{n\in B} c_n \chi_{K_n}\Big\|_p\le \sum\limits_{n\in B} c_n \|\chi_{K_n}\|_p=\sum\limits_{n\in B} c_n \lambda(K_n)^{\frac{1}{p}}<\infty.$$
From the above inequality, the inclusion $y\in\overline{span\{x_n: n\in \mathbb{N}\}}$ readily follows.

On the other hand, if $n\in B$, then
\begin{align*}
T_{a,w}^n y&=\sum \limits_{k\in B}c_k T_{a,w}^n \chi_{K_k}\\
&=\sum \limits_{k\in B}c_k T_{a^n}(\varphi_n \chi_{K_k})\\
&\geq T_{a^n}\left(c_n\varphi_n \chi_{K_n}\right).
\end{align*}
This implies that $\|T_{a,w}^n y\|_p\geq c_n\|\varphi_n|_{K_n}\|_p \to \infty$ as $n\to \infty$ and $n\in B$.
Thus, $T_{a,w}$ satisfies the condition (ii) of (DCC).

Therefore, $T_{a,w}$ is distributionally chaotic.
\end{proof}



\subsection{Existence of distributionally chaotic weighted translation} 
\label{sec:existence_of_distributionally_chaotic_weighted_translation}
In this subsection, we will prove that for an ``aperiodic'' element $a\in G$,
there always exists a weight $w$ such that $T_{a,w}$ is distributionally chaotic.

We say an element $a\in G$ is aperiodic if
the cyclic subgroup generated by $a$ is \textbf{not} precompact in $G$.
The aperiodicity of the element $a$ plays a crucial role in multiple sense of chaoticity for a weighted translation $T_{a,w}$
(see more discussion in \cite{cc11,kuchen17}).

Recall that any locally compact groups are compatible with a proper right invariant metric $d$
(see more detail on \cite[Section 2]{kuchen17} and \cite{plig}).
In this metric space, the Heine-Borel property holds.
Therefore, $B_r(x):=\{y\in G|d(x,y)<r\}$ is a precompact open ball, and $B_r(xa) = B_r(x)a$ by right invariance.

\begin{lemma}
\label{aperiodic_lemma}
Let $a\in G$ be aperiodic.
Then there exist a compact neighborhood $K_0$ of the identity element $e\in G$ and $\varepsilon >0$ such that
$d(K_0a^n, K_0a^m) \ge \varepsilon$ for any distinct pair $n,m\in\mathbb{Z}$.
In particular, any union of subcollection of $\{K_0a^n\}_{n\in\mathbb{Z}}$ is a closed set in $G$.
\end{lemma}
\begin{proof}
By \cite[Theorem 2.9]{kuchen17}, $d(a^n, e)\to \infty$ as $n\to \infty$.
So $\varepsilon :=
\inf\limits_{n\in\mathbb{Z}\setminus \{0\}} \frac{d(a^n, e)}{10}>0$

Let $r \in (0, \varepsilon)$.
Define the compact neighborhood $K_0$ of $e$ to be the closure of $B_r(0)$.
Then for any $g_1, g_2\in K_0$ and $n\in\mathbb{Z}\setminus \{0\}$,
we claim that $d(g_1a^n, g_2) \ge \varepsilon$.
If $d(g_1a^n, g_2) < \varepsilon$, then
\begin{align*}
d(a^n, e)
&\le
d(g_1a^n, g_2) + d(g_1a^n, a^n) + d(g_2, e)\\
&\le
\varepsilon + r + r\\
&<
3\varepsilon\\
&<
d(a^n, e).
\end{align*}
So there is a contradiction.

Thus, for any distinct pair $n,m\in\mathbb{Z}$,
we have $d(K_0a^n, K_0a^m) = d(K_0a^{n-m}, K_0) \ge \varepsilon$.
\end{proof}
In this section, we will use this compact neighborhood $K_0$ to construct a weight $w$ and show that $T_{a,w}$ is distributionally chaotic.

The next result reveals that for any $m\in\mathbb{Z}$,
the dynamic behavior of the $T_{a,w}$-action on $T_{a^m}y$ is controlled by
the $T_{a,w}$-action on $y$ and vice versa.
\begin{lemma}\label{convergent_and_divergent_for_any_subsequence}
Let $T_{a,w}$ be an invertible weighted translation operator on $L^p(G)$.
Then for any $y\in L^p(G)$, $n, m\in\mathbb{Z}$ with $n$ large enough (explicitly, $n > 2|m|$), we have
$$M^{-2|m|}\|T_{a,w}^n y\|_p \le \|T_{a,w}^n (T_{a^m} y)\|_p\le M^{2|m|}\|T_{a,w}^n y\|_p$$
for some $M\ge 1$.
\end{lemma}
\begin{proof}
Let $M=\max\{\sup w, (\inf w)^{-1}\}$. Then $M<\infty$ by the invertibility of $T_{a,w}$. Rewrite
$T_{a^{-m}}\varphi_n=\prod \limits_{s=1+m}^{n+m}T_{a^{-s}}w$ as
$$\left\{\begin{matrix}
&\varphi_n
&\cdot &\prod \limits_{s=1}^{m}T_{a^{-s}}(w^{-1})
&\cdot &\prod \limits_{s=n+1}^{n+m}T_{a^{-s}}w
&\text{if } n>>m>0; \\
&\varphi_n
&\cdot &\prod \limits_{s=1+m}^{0}T_{a^{-s}}w
&\cdot &\prod \limits_{s=n+m+1}^{n}T_{a^{-s}}(w^{-1})
&\text{if } n>>0>m.
\end{matrix}\right.$$
Then we have that
$$M^{-2|m|}\varphi_n \le T_{a^{-m}}\varphi_n \le M^{2|m|}\varphi_n$$
for $n$ large enough.

Hence,
$$M^{-2|m|}\|T_{a,w}^n y\|_p \le \|T_{a,w}^n (T_{a^m} y)\|_p\le M^{2|m|}\|T_{a,w}^n y\|_p$$
by the computations
$$T_{a,w}^n y=T_{a^n} M_{\varphi_n} y$$
and
$$T_{a,w}^n T_{a^m} y=T_{a^n} M_{\varphi_n} T_{a^m} y=T_{a^{n+m}}M_{T_{a^{-m}}\varphi_n} y.$$
\end{proof}

\begin{theorem}
\label{existence}
Let $a\in G$ be an aperiodic element. Then
there always exists a weight $w$ such that $T_{a,w}$ is distributionally chaotic on $L^p(G)$ for any $p\in [1,\infty)$.
In particular, this operator $T_{a,w}$ is Li-Yorke chaotic on $L^p(G)$ for any $p\in [1,\infty)$.
\end{theorem}
\begin{proof}
Define
\begin{align*}
I_k := & ((2k-1)!, (2k)!] \cap \mathbb{N}, \\
J_k := & ((2k)!, (2k+1)!] \cap \mathbb{N}, \\
A' :=  & \bigcup_{k=1}^\infty I_k, \\
A''  :=  & \bigcup_{k=1}^\infty J_k
\end{align*}
and
$w\equiv 2$ on $\bigcup\limits_{n\in A'} K_0a^n$,
$w\equiv \frac{1}{2}$ on $\bigcup\limits_{n\in A''} K_0a^n$.
By Urysohn's lemma, $w$ is extended continuously to $G$.

So far, we define the operator $T_{a,w}$.
We are going to check that $T_{a,w}$ satisfies (DCCW).

Let $g_n:=a^{(2n-1)!}$.
Then
$\varphi_n(g_n) = \prod \limits_{j=1}^n w(a^{(2n-1)! + j})=2^n$
for all $n\in \mathbb{N}$.
So $\sum\limits_{n\in B} \frac{1}{\varphi_n(g_n)}$ always converges for any full upper density $B$.

Let $K_k:=K_0 g_k$ be a compact neighborhood of $g_k$.
Then by Lemma \ref{convergent_and_divergent_for_any_subsequence}, we have
\begin{align*}
\|\varphi_n|_{K_k}\|_p
&=
\|T_{a,w}\chi_{K_k}\|_p\\
&=
\|T_{a,w}T_{a^{(2n-1)!}}(\chi_{K_0})\|_p\\
&\le
M^{2(2k-1)!}\|T_{a,w}\chi_{K_0}\|_p\\
&=
M^{2(2k-1)!}\|\varphi_n\chi_{K_0}\|_p
\end{align*}
for some $M \ge 1$ and $n > 2(2k-1)!$, where $\chi_{K_k}$ and $\chi_{K_0}$ are the characteristic functions of $K_k$ and $K_0$.

Define $A := \bigcup_{\tau =1}^\infty (2(2\tau )!,(2\tau +1)!] \cap \mathbb{N} \subset A''$.

\noindent\textbf{Claim:}
$\displaystyle\lim_{n\in A}\|\varphi_n\chi_{K_0}\|_p=0$.
Thus, $\displaystyle\lim_{n\in A}\|\varphi_n\chi_{K_k}\|_p=0$ for all $k$.

Suppose that $n \in (2(2\tau )!,(2\tau +1)!]$ and $\tau >4$. We have
\begin{align*}
(\varphi_n\chi_{K_0}) (x)
&=
\prod_{j=1}^{n}w(x a^j)\chi_{K_0}(x)\\
&=
\prod_{j=1}^{(2\tau -2)!} w(x a^j)\chi_{K_0}(x)
\cdot
\prod_{j=(2\tau -2)! + 1}^{(2\tau -1)!} w(x a^j)\chi_{K_0}(x)\\
&
\cdot
\prod_{j=(2\tau -1)! + 1}^{(2\tau )!} w(x a^j)\chi_{K_0}(x)
\cdot
\prod_{j=(2\tau )! + 1}^{2(2\tau )!} w(x a^j)\chi_{K_0}(x)\\
&
\cdot
\prod_{j=2(2\tau )! + 1}^{n} w(x a^j)\chi_{K_0}(x)
\\
&\le
2^{(2\tau -2)!}
\cdot
2^{-((2\tau -1)! - (2\tau -2)!)}\\
&
\cdot
2^{(2\tau )! - (2\tau -1)!}
\cdot
2^{-(2\tau )!}\\
&
\cdot
\chi_{K_0}(x)\\
&=
2^{-2((2\tau -1)! - (2\tau -2)!)}\chi_{K_0}(x).
\end{align*}
Thus, $\displaystyle\lim_{n\in A}\|\varphi_n\chi_{K_0}\|_p=0$.

$ $

\noindent\textbf{Claim:}
$\overline{{\rm dens}}(A)=1$.

$
1
\ge
\overline{{\rm dens}}(A)
\ge
\displaystyle\lim_{\tau\to\infty} \frac{(2\tau +1)! - 2(2\tau )!}{(2\tau +1)!}=1$.
So $\overline{{\rm dens}}(A)=1$.

Thus, $T_{a,w}$ satisfies (DCCW).
\end{proof}

\begin{remark}
\label{existence_rmk}
For the construction of the weight $w$ in Theorem \ref{existence},
we can similarly utilize arguments by interchanging of $I_k$ and $J_k$ to show that
$T_{a,w^{-1}}$ is also distributionally chaotic.
We will recall this in Example \ref{counter_ex}.
\end{remark}

Similarly applying the proof of theorem above to bilateral weighted shifts,
we set $G=\mathbb{Z}$ and $a=-1$, and consider a weight function $w$ defined on $\mathbb{Z}$.
Then the weighted translation $T_{-1,w}$ is nothing but a bilateral weighted shift on $\ell^p(\mathbb{Z})$ denoted by $B_w:=T_{-1,w}$.
\begin{example}\label{ex_Bw}
Define $w:\mathbb{Z}\to \mathbb{R}$ by
$$
w(n):=
\left\{\begin{matrix}
2
&\text{if } (2k-1)!< |n| \le (2k)! \text{ for some } k\in\mathbb{N}; \\
\frac{1}{2}
&\text{if } (2k)!< |n| \le (2k+1)! \text{ for some } k\in\mathbb{N}; \\
1
&|n|\le 1.
\end{matrix}\right.$$
Then $B_w$, $B_{w^{-1}}$ and there inverse are all distributionally chaotic on $\ell^p(\mathbb{Z})$ for any $p\in [1, \infty)$.
\end{example}

\section{Distributionally irregular vectors}
\label{sec:distributionally_irregular_vectors}

In this section, we will recall the definition of distributionally irregular vectors and
discuss some properties (Theorem \ref{translated irregular vector}) of distributionally irregular vectors for weighted translations on locally compact groups.

\begin{definition}{\rm(\cite[Definition 3]{2013JFA})}\label{DIV}
Let $T$ be an operator on a Banach space $X$. Then $x\in X$ is called a {\it distributionally irregular vector} for $T$
if there exist $A,B\subset \mathbb{N}$ with $\overline{\mbox{dens}}(A)=\overline{\mbox{dens}}(B)=1$ such that
$$\lim_{n\in A}T^nx=0\ \ \mbox{and}\ \ \lim_{n\in B}\|T^nx\|=\infty.$$
\end{definition}

If $x\in X$ only satisfies the first condition in Definition \ref{DIV} i.e. $\displaystyle\lim_{n\in A}T^nx=0$ for some $A\subseteq \mathbb{N}$ with $\overline{\mbox{dens}}(A)=1$,
then the orbit of $x$ is said to be {\it distributionally near to $0$}.
On the other hand, $x$ has a {\it distributionally unbounded orbit}
if $\displaystyle\lim_{n\in B}\|T^nx\|=\infty$ for some $B\subseteq \mathbb{N}$ with $\overline{\mbox{dens}}(B)=1$.
Hence $x\in X$ is distributionally irregular if, and only if,
its orbit is both distributionally near to $0$ and distributionally unbounded.

For further study,
we denote the set of all distributionally irregular vectors of an operator $T$ by $DIV(T)$,
and define
$$X_{0,A}(T):=\left\{x\in X: \lim_{n\in A}T^nx=0\right\},$$
$$DIV_{A,B}(T):=\left\{x\in X: \lim_{n\in A}T^nx=0\ \ \mbox{and}\ \ \lim_{n\in B}\|T^nx\|=\infty\right\}$$
and
$$DIV_{A,\bullet}(T):=DIV(T)\cap X_{0,A}(T)=\bigcup_{\overline{dens}(B')=1} DIV_{A,B'}(T)$$
for fixed $A,B\subset \mathbb{N}$ with $\overline{\mbox{dens}}(A)=\overline{\mbox{dens}}(B)=1$.
If the given operator $T$ is clear,
then we write $DIV$, $X_{0,A}$, $DIV_{A,B}$ and $DIV_{A,\bullet}$
for simplification.

Throughout this section, we will show that some crucial results for distributionally irregular vectors of $T_{a,w}$.
Suppose that $DIV_{A,B}$ is non-empty. Then we have the following.
\begin{enumerate}
\item
$DIV_{A,B}$ is invariant under some operations (see Theorem \ref{translated irregular vector}).

\item
$DIV_{A,B}$ contains a cone structure $\cal{C}$ (see Theorem \ref{cone}).

\item
$DIV_{A,B}$ and $DIV_{A,\bullet }$ are both path connected
(see Corollary \ref{ab_path} and Theorem \ref{DIV_A_path_connected}).

\item
There exists some unique (up to measure zero) measurable subset $M_A$ in $G$ which is independent of $B$ such that
$$
\overline{DIV_{A,\bullet}}=\overline{DIV_{A,B}}=L^p(M_A)=\overline{X_{0,A}},
$$
(see Theorem \ref{m_ab}).

Moreover,
$$M_A a^n\subseteq M_A \text{ for } n\in \mathbb{N}.$$
In addition, if $T_{a,w}$ is invertible, then
$$M_A a^n = M_A \text{ for } n\in \mathbb{Z},$$
(see Theorem \ref{m_a}).

\item
$DIV \cap L^p(M_A)$ is connected (see Theorem \ref{m_ab}).
\end{enumerate}

In the final example (Example \ref{counter_ex}),
we give an example of distributionally chaotic $T_{a,w}$ that is related to the construction in Theorem \ref{existence}.
For this $T_{a,w}$, we describe explicitly that $DIV$ has only two connected components
and $\overline{DIV}$ is the union of two subspaces $L^p(M_A)$ and $L^p(M_{A'})$ for some disjoint $M_A$ and $M_{A'}$.

\begin{theorem}\label{translated irregular vector}
Let $T_{a,w}$ be a weighted translation on $L^p(G)$.
Then the following conditions are equivalent.
\begin{enumerate}
\item[{\rm(i)}]
$y\in DIV_{A,B}$.
\item[{\rm(ii)}]
$|y|\in DIV_{A,B}$.
\end{enumerate}

In addition, if $T_{a,w}$ is invertible,
then the above assertions are also equivalent to the following.
\begin{enumerate}
\item[{\rm(iii)}]
$T_{a^m} y\in DIV_{A,B}$ for some $m\in \mathbb{Z}$.
\item[{\rm(iv)}]
$T_{a^m} y\in DIV_{A,B}$ for all $m\in \mathbb{Z}$.
\end{enumerate}
\end{theorem}
\begin{proof}
Since
$$\|T_{a,w}^n y\|_p^p=\|\varphi_n y\|_p^p=\int_G \varphi_n^p |y|^p=\|T_{a,w}^n |y|\|_p^p,$$
conditions (i) and (ii) are equivalent. By Lemma \ref{convergent_and_divergent_for_any_subsequence},
conditions (iii) and (iv) are also equivalent to (i).
\end{proof}


Let $A\subset \mathbb{N}$ be as in the definition of the distributionally irregular vector $y$.
Notice that $\overline{X_{0,A}}$ is an invariant subspace of $T_{a,w}$ with $y\in\overline{X_{0,A}}$
(see the argument in the proof (i) $\Rightarrow$ (ii) of \cite[Theorem 12]{2013JFA}).
Then the set of all distributionally irregular vectors of $T_{a,w}$ in $\overline{X_{0,A}}$ is a residual set in $\overline{X_{0,A}}$
by \cite[Proposition 7 and Proposition 9]{2013JFA}.

In particular, for the case of invertible weighted shifts $B_w:=T_{-1,w}$, we will show that the set of all distributionally irregular vectors of $B_w$ is residual in $\ell^p(\mathbb{Z})$ using Theorem \ref{translated irregular vector}.

\begin{theorem}\label{residual}
Let $B_w$ be an invertible weighted shift on $\ell^p(\mathbb{Z})$.
If $B_w$ is distributionally chaotic,
then the set of all distributionally irregular vectors of $B_w$ is residual in $\ell^p(\mathbb{Z})$.
\end{theorem}
\begin{proof}
If $B_w$ is distributionally chaotic,
then $B_w$ has a distributionally irregular vector $y$ by \cite[Theorem 12]{2013JFA}.
Let $A$ be defined as in the definition of the distributionally irregular vector $y$ and
$$X_{0,A}:=\left\{x: \lim\limits_{n\in A}B_{w}^n x=0\right\}.$$

Let $g\in \mathbb{Z}$ with $y(g)\neq 0$, and let $\chi_g$ be the characteristic function of $g$.
Since $|y(g)|\chi_g \le |y| \in X_{0,A}$, we have $\lim\limits_{n\in A}B_{w}^n |y(g)|\chi_g=0,$
which implies $\chi_g\in X_{0,A}$. By Theorem \ref{translated irregular vector}, for any $m\in\mathbb{Z}$, the vector $T_m\chi_g$ belongs to $X_{0,A}$ as well.
So $X_{0,A}$ contains the standard basis of $\ell^p(\mathbb{Z})$.
Hence, $\overline{X_{0,A}}=\ell^p(\mathbb{Z})$. Therefore, the set of all distributionally irregular vectors of $B_w$ is residual in $\ell^p(\mathbb{Z})$.
\end{proof}

\begin{remark}
Let $B_w$ be distributionally chaotic.
We can construct $T_{-2, w'}$ on $\ell^p(\mathbb{Z})$ by the following:
$$w'(n):=
\left\{\begin{matrix}
w(\frac{n}{2}) & \text{if $n$ is even};\\
1    & \text{if $n$ is odd}.
\end{matrix}\right.$$
Then $T_{-2, w'}$ acts on $\ell^p(\mathbb{Z})=\ell^p(2\mathbb{Z})\oplus \ell^p(2\mathbb{Z}+1)$ with
a distributionally irregular vector $y$ with support in $2\mathbb{Z}$.

Let $A$ be as in the definition of the distributionally irregular vector $y$.
We have $\overline{X_{0,A}}=\ell^p(2\mathbb{Z})\neq \ell^p(\mathbb{Z})$.
We see that
$\overline{X_{0,A}}$ would not necessary equal to $L^p(G)$ for weighted translations $T_{a,w}$
in general.
\end{remark}





Next, we will show that $DIV_{A,B}(T_{a,w})$ contains a cone structure.

\begin{theorem}\label{cone}
Let $T_{a,w}$ be a distributionally chaotic weighted translation operator on $L^p(G)$.
Consider a non-empty $DIV_{A,B}$.
Define the cone $\cal{C}$ by
$$\cal{C}:=\left\{\sum\limits_{\text{finite }j}c_j |y_j|: c_j> 0 \text{, where $y_i\in DIV_{A,B}$}\right\}.$$
Then each nonzero vector in the set $\cal{C}$ is also a distributionally irregular vector.
Furthermore, if $T_{a,w}$ is invertible, then $\cal{C}$ is invariant under these operators $T_a$, $T_{a,w}$, $M_w$, $M_{T_{a^n} w}$, $M_{\varphi_n}$ and their inverses.
\end{theorem}
\begin{proof}

First, observe that each vector of $\cal{C}$ is distributionally near to $0$ w.r.t. $A$.
On the other hand, the vector $\sum \limits_{\text{finite }j}c_j |y_j|$ is also distributionally unbounded w.r.t. $B$
by the inequality $\sum\limits_{\text{finite }j}c_j |y_j|\ge c_{j_0} |y_{j_0}|,$ where
$|y_{j_0}|$ is a distributionally irregular vector for some fixed $j_0$. Combing all of these, one obtains that each element in $\cal{C}$ is in $DIV_{A,B}$.

For the invariance of $\cal{C}$, we only need to prove the two cases $T_a$ and $T_{a,w}$ since the others are just the composition of them.
Indeed, $\cal{C}$ is invariant under $T_{a,w}$ by the inequality
$$m\|T_{a,w}^n y\|_p\le\|T_{a,w}^{n+1} y\|_p\le M\|T_{a,w}^n y\|_p$$
for some constant $m, M>0$.
Moreover, the invariance of $\cal{C}$ under $T_a$ can be deduced by Theorem \ref{translated irregular vector}, which relies on the invertibility of $T_{a,w}$.
\end{proof}

\begin{remark}
We note that $\cal{C}$ is path connected because it is convex.
\end{remark}

Let $y$ be a distributionally irregular vector of $T_{a,w}$.
Construct the equivalence class $[y]$ of $y$ by
$$[y]:=\{z: c_1|z|\le |y|\le c_2|z|\ \mbox{for some}\ c_1,c_2>0\ \mbox{which only depend on}\ y,z\}.$$
Then all elements in $[y]$ are distributionally irregular vectors of $T_{a,w}$.

\begin{theorem}\label{equivalence_relation}
Let $T_{a,w}$ be a weighted translation operator on $L^p(G)$, and let $y$ be a distributionally irregular vector of $T_{a,w}$. Then we have the following.
\begin{enumerate}
\item[{\rm(i)}]
All vectors in the equivalence class $[y]$ are distributionally irregular vectors.
\item[{\rm(ii)}]
The class $[y]$ is path connected.
\item[{\rm(iii)}]
The class $[y]$ is dense in the closed subspace $L^p(\{y\neq 0\})$, i.e. $$\overline{[y]} = L^p(\{y\neq 0\}).$$
In particular, the set $DIV\cap L^p(\{y\neq 0\})$ is connected,
where $\{y\neq 0\}:=\{x\in G : y(x)\neq 0\}$.
\end{enumerate}
\end{theorem}
\begin{proof}
(i) Let $z\in [y]$. Then, by the inequality $c_1|z|\le |y|\le c_2|z|$, $z$ is a distributionally irregular vector.

(ii) Let $z\in [y]$. We define a path $p':[0,1]\to [y]$ from $|z|$ to $|y|$ by
$$p'(t)=t|y|+(1-t)|z|.$$
Then $p'$ is continuous since $t(|y|-|z|)$ converges to $0$ in the $p$-norm as $t\to 0$ by the dominated convergence theorem.
On the other hand, let $g\in G$, $r_g=|z(g)|$ and $\theta_g=\mbox{Arg}\ z(g)$. Then $z(g)=r_g e^{i\theta_g}$.
Define another path $p'':[0,1]\to [y]$ from $|z|$ to $z$ by
$$p''(t)(g)=r_g e^{it\theta_g}$$
where $|p''(t)|=|z|$ for all $t$. Hence the class $[y]$ is path connected by $p'$ and $p''$.

(iii) Let $$S_{y, \alpha, \beta}=\{x: \alpha< |y(x)|< \beta\}$$
where $0 < \alpha < 1 < \beta < \infty$.
Let $\zeta \in L^p(S_{y, \alpha, \beta})$ be a simple function.
Define $m:=\inf_{\zeta(g)\neq 0}|\zeta(g)|$ and $M:=\sup_{g\in G}|\zeta(g)|$.
By the definition of simple function, we have $0<m \le M<\infty$.
Then one can find $\varepsilon_0:=\min\{\frac{m}{2\beta},1\} >0$ such that $\zeta +\varepsilon y\in [y]$ for $\varepsilon\in (0,\varepsilon_0 )$. Thus $\zeta\in \overline{[y]}$.

To see this, let $c_1=\varepsilon $ and $c_2=\max\{\frac{M}{\alpha},1\}+1$.
If $g\in \{\zeta = 0\}$, then $c_1|y|(g)\le |\zeta+\varepsilon y|(g)\le c_2|y|(g)$.
On the other hand, for $g\in \{\zeta \neq 0\}\subseteq S_{y, \alpha, \beta}$,
$$2\varepsilon |y(g)|\le m\frac{|y(g)|}{\beta}\le m\le |\zeta(g)|.$$
So we have
\[\varepsilon |y(g)|
\le |\zeta(g)|-\varepsilon |y(g)|
\le |(\zeta+\varepsilon y)(g)|
\le |\zeta(g)|+|y(g)|
\le (\frac{M}{\alpha}+1)|y(g)|
\le c_2|y(g)|
.\]
Hence, $c_1|y|\le |\zeta+\varepsilon y|\le c_2|y|$.
Thus, $\zeta +\varepsilon y\in [y]$.

Next, let $\varphi \in L^p(\{y\neq 0\})$ be a simple function.
Consider
$$\zeta_n:=\varphi \chi_{S_{y, \frac{1}{n},n}}\chi_{\{|\varphi|>\frac{1}{n}\}}\in L^p(S_{y, \frac{1}{n},n}).$$
By the discussion above, $\zeta_n\in \overline{[y]}$ and $\zeta_n\to \varphi$ in $p$ norm as $n\to \infty$.
Hence, $\varphi\in \overline{[y]},$ which says that $\overline{[y]}= L^p(\{y\neq 0\})$.

Finally, since $[y]$ is path connected and dense in $L^p(\{y\neq 0\})$, the fact
$$[y]\subseteq DIV\cap L^p(\{y\neq 0\}) \subseteq L^p(\{y\neq 0\})$$
implies that $DIV\cap L^p(\{y\neq 0\})$ is connected.
\end{proof}

\begin{remark}
In Theorem \ref{equivalence_relation}, we cannot replace $\{y\neq 0\}$ by $supp(y)$ since the measure of the boundary of $\{y\neq 0\}$ could be positive. For instance, the boundary of $\{\chi_U\neq 0\}$ has the positive measure, where $U$ denotes the complement of Cantor-like set in $[0,1]$.

\end{remark}

\begin{corollary}\label{ab_path}
Let $T_{a,w}$ be a distributionally chaotic weighted translation operator on $L^p(G)$.
Then the set $DIV_{A,B}$ is path connected.
\end{corollary}
\begin{proof}
If $DIV_{A,B} = \emptyset$, then there is nothing need to say.
If $y,y'\in DIV_{A,B}$, then we have
$|y|\in [y]\cap \cal{C}$
and
$|y'|\in [y']\cap \cal{C}$.
On the other hand, since $[y]$, $[y']$ and $\cal{C}$ are path connected subsets of $DIV_{A,B}$.
So there is a path in $DIV_{A,B}$ passing through $|y|$ and $|y'|$ that connects $y$ and $y'$.
Thus, $DIV_{A,B}$ is path connected.
\end{proof}

\begin{corollary}\label{countable_union}
Let $T_{a,w}$ be a weighted translation operator on $L^p(G)$ and let $\{y_i\}_{i=1}^{\infty}\subseteq DIV_{A,B}$. Then $DIV_{A,B}\cap L^p(\bigcup\limits_{i=1}^{\infty}\{y_i\neq 0\})$ is dense in $L^p(\bigcup\limits_{i=1}^{\infty}\{y_i\neq 0\})$. In particular, $DIV \cap L^p(\bigcup\limits_{i=1}^{\infty}\{y_i\neq 0\})$ is connected.
\end{corollary}
\begin{proof}
Let $y_N':=\sum\limits_{i=1}^N|y_i|$. By the fact $y_N'\in \cal{C}$, we get that $y_N'\in DIV_{A,B}$.
Let $\zeta \in L^p(\bigcup\limits_{i=1}^{\infty}\{y_i\neq 0\})$ be bounded and compact support.
Consider
$$f_N:=\zeta\chi_{\{y_N'\neq 0\}}\in L^p(\{y_N'\neq 0\})=\overline{[y_N']}\subseteq \overline{DIV_{A,B}\cap L^p(\bigcup\limits_{i=1}^{\infty}\{y_i\neq 0\})}.$$
Then
$\zeta\in \overline{DIV_{A,B}\cap L^p(\bigcup\limits_{i=1}^{\infty}\{y_i\neq 0\})}$
since $\zeta$ is a limit of $\{f_N\}$.
Hence
$$L^p(\bigcup\limits_{i=1}^{\infty}\{y_i\neq 0\})=\overline{DIV_{A,B}\cap L^p(\bigcup\limits_{i=1}^{\infty}\{y_i\neq 0\})}.$$

To see that $DIV\cap L^p(\bigcup\limits_{i=1}^{\infty}\{y_i\neq 0\})$ is connected, it is sufficient to show that $DIV_{A,B} \cap L^p(\bigcup\limits_{i=1}^{\infty}\{y_i\neq 0\})$ is path connected. Let $z,z'\in DIV_{A,B}\cap L^p(\bigcup\limits_{i=1}^{\infty}\{y_i\neq 0\})$.
As in the proof of Corollary \ref{ab_path}, we can find a path starting from $z$ to $z'$ through $|z|$ and $|z'|$.
\end{proof}

\begin{remark}
It should be noted that in the above corollary we cannot state the stronger statement
$$L^p(\bigcup_{y\in DIV_{A,B}}\{y\neq 0\})=\overline{DIV_{A,B}},$$
which looks more natural.
This is because the union $\bigcup\limits_{y\in DIV_{A,B}}\{y\neq 0\}$ is uncountable, and the measurability could be doubtful.
But there is a way in which we can improve the result of Corollary \ref{countable_union}
from
$$L^p(\bigcup\limits_{i=1}^{\infty}\{y_i\neq 0\})=\overline{DIV_{A,B}\cap L^p(\bigcup\limits_{i=1}^{\infty}\{y_i\neq 0\})}$$
to
$$L^p(\bigcup\limits_{i=1}^{\infty}\{y_i\neq 0\})=\overline{DIV_{A,B}},$$
by selecting the sequence $\{y_i\}_{i=1}^{\infty}$ properly
(see Theorem \ref{m_ab}).
\end{remark}

\begin{theorem}
\label{DIV_A_path_connected}
If $DIV_{A,B}(T_{a,w})$ is nonempty,
then $DIV_{A,\bullet }(T_{a,w})$ is path connected.
\end{theorem}
\begin{proof}
According to Corollary \ref{ab_path}, $DIV_{A,B}$ is path connected.
Moreover,
$$DIV_{A,\bullet}=DIV\cap X_{0,A}=\bigcup_{\overline{dens}(B')=1}DIV_{A,B'}.$$
So it is sufficient to show that $DIV_{A,B'}\cap DIV_{A,B''}\neq \emptyset$ if both $DIV_{A,B'}$ and $DIV_{A,B''}$ are nonempty.

Let $y'\in DIV_{A,B'}$ and $y''\in DIV_{A,B''}$. Then $y:=|y'|+|y''|$ is in $X_{0,A}$. Also, $y$ is  distributionally unbounded w.r.t. $B'$ and $B''$
by $|y'|\le y$ and $|y''|\le y$ respectively. Hence, $y\in DIV_{A,B'}\cap DIV_{A,B''}$.
\end{proof}

\begin{lemma}
\label{separable}
Let $X$ be a second countable space, and let $Y$ be an arbitrary subset of $X$. Then there exists a sequence $\{y_i\}_{i=1}^{\infty}$ in $Y$ such that the sequence is dense in $Y$. Hence, $\overline{\{y_i\}}=\overline{Y}$ in $X$.
\end{lemma}
\begin{proof}
The subspace $Y$ is separable since it is second countable.
\end{proof}

\begin{theorem}
\label{m_a}
There exists a unique (up to measure zero) measurable set $M_A\subseteq G$ such that the closed subspace $\overline{X_{0,A}(T_{a,w})}$ is of the form $L^p(M_A)$.

Moreover,
$$M_A a^n\subseteq M_A \text{ for } n\in \mathbb{N}.$$
In addition, if $T_{a,w}$ is invertible, then
$$M_A a^n = M_A \text{ for } n\in \mathbb{Z}.$$
\end{theorem}
\begin{proof}
Applying Lemma \ref{separable} with $X=L^p(G)$ and $Y=X_{0,A}$,
there exists $\{y_i\}_{i=1}^{\infty}$ in $X_{0,A}$ such that $\overline{\{y_i\}}=\overline{X_{0,A}}$.
Let $M_A:=\bigcup_i \{y_i\neq 0\}$. Then $\overline{X_{0,A}}=\overline{\{y_i\}}\subseteq L^p(M_A)$.

On the other hand, let $M$ be a subset of $M_A$ with finite measure,
and let
$$M_{i,k}:=M\cap \left(\bigcup_{\ell=1}^i\{y_{\ell}\neq 0\}\setminus \bigcup_{\ell=1}^{i-1}\{y_{\ell}\neq 0\}\right)\cap \Biggl\{\frac{1}{k}<|y_i|\le\frac{1}{k-1}\Biggr\},$$
where $\frac{1}{k-1}:=\infty$ when $k=1$ for convenience.
Then $M=\coprod_{i,k}M_{i,k}$ is a countable disjoint union.
Moreover, by $0\le \frac{1}{k}\chi_{M_{i,k}}\le |y_i|$,
we have $\chi_{M_{i,k}}\in X_{0,A}$ which implies $\chi_M\in \overline{X_{0,A}}$.
Therefore, all the simple functions of $L^p(M_A)$ are contained in $\overline{X_{0,A}}$.

Combing all of these, $\overline{X_{0,A}}=L^p(M_A)$,
and the uniqueness follows by the fact that $L^p(M')=L^p(M'')$ implies $M'=M''$ almost everywhere.

Next, we will show that $M_A a^n\subseteq M_A \text{ for } n\in \mathbb{N}.$
Since $X_{0,A}$ is invariant under $T_{a,w}$,
we have $T_{a,w}y_i\in X_{0,A} \subseteq L^p(M_A).$
This means that $\{y_i\neq 0\}a=\{T_{a,w}y_i\neq 0\}\subseteq M_A$.
Thus, $M_A a=\bigcup_i\{y_i\neq 0\}a\subseteq M_A$.
If $T_{a,w}$ is invertible, then $X_{0,A}$ is also invariant under $T_{a,w}^{-1}$.
Hence, $M_A a^n = M_A$.
\end{proof}

\begin{theorem}
\label{m_ab}
Suppose that $DIV_{A,B}(T_{a,w})$ is nonempty in $L^p(G)$.
We have
$$
\overline{DIV_{A,\bullet}}=\overline{DIV_{A,B}}=L^p(M_A)=\overline{X_{0,A}},
$$
where $M_A$ is induced by Theorem \ref{m_a}.

In particular, the closure of $DIV$ is a union of closed subspaces of the form $L^p(M)$ for some $M$s.
Furthermore, by Corollary \ref{countable_union}, $DIV \cap L^p(M_A)$ is connected.

\end{theorem}
\begin{proof}
By setting $X=L^p(G)$ and $Y=DIV_{A,B}$ in Lemma \ref{separable}, there exists $\{y_i\}_{i=1}^{\infty}$ in $DIV_{A,B}$ so that $\overline{\{y_i\}}=\overline{DIV_{A,B}}$.
Let $M_{A,B}:=\bigcup_i \{y_i\neq 0\}$.
Then by Corollary \ref{countable_union}, we have
$$\overline{DIV_{A,B}}=\overline{\{y_i\}}\subseteq L^p(M_{A,B})=\overline{DIV_{A,B}\cap L^p(M_{A,B})}\subseteq \overline{DIV_{A,B}}.$$
Hence $\overline{DIV_{A,B}}=L^p(M_{A,B})$.

By the assumption that $DIV_{A,B}$ is nonempty,
we may choose $y\in DIV_{A,B}$ to be nonnegative.
Assume that $K:=M_A\setminus M_{A,B}$ is non-null.
Using similar argument in the proof of Theorem \ref{m_a},
we may assume
$K=\coprod_j M_j$ is a countable disjoint union such that $\chi_{M_j}\in X_{0,A}$.
Set $y'_j:=\chi_{M_j}+y$.
Since both $\chi_{M_j}, y\in X_{0,A}$, so $y'_j\in X_{0,A}$.

On the other hand, $y'_j$ is distributionally unbounded w.r.t. $B$ by $y\le y'_j$.
So $y'_j\in DIV_{A,B}\subset L^p(M_{A,B})$.
But, this implies that $K=\bigcup_j M_j\subseteq\bigcup_j \{y'_j\neq 0\}\subseteq M_{A,B}$ almost everywhere,
which contradicts to the assumption that $K=M_A\setminus M_{A,B}$ is non-null.
Hence $M_A=M_{A,B}$ almost everywhere,
and
$$\overline{DIV_{A,B}}=L^p(M_{A,B})=L^p(M_{A})=\overline{X_{0,A}}.$$

Finally, since $DIV_{A,B} \subseteq DIV_{A,\bullet} \subseteq X_{0,A}$, 
$$
\overline{DIV_{A,B}}=\overline{DIV_{A,\bullet}}=L^p(M_A)=\overline{X_{0,A}}.
$$
\end{proof}


Comparing with Theorem \ref{residual}, we use a different method to obtain the following result.

\begin{corollary}
Let $B_w$ be invertible and distributionally chaotic on $\ell^p(\mathbb{Z})$. Then $DIV(B_w)$ is dense and connected in $\ell^p(\mathbb{Z})$.
\end{corollary}
\begin{proof}
Let $y\in DIV_{A,B}(B_w)$ for some $A,B$, and recall that $B_w=T_{a,w}$ when $a=-1$.
Then the only nonempty subset $M_A$ of $\mathbb{Z}$ satisfying $M_A a^n = M_A$ for all $n$ (or rewrite it in the form of the addition group, that is, $M_A -n = M_A $ for $n\in \mathbb{Z}$) is $\mathbb{Z}$ itself.
Thus, $M_A=\mathbb{Z}$. Hence, $DIV$ is dense in $\ell^p(M_A)=\ell^p(\mathbb{Z})$ by Theorem \ref{m_ab}.

On the other hand, since $M_A$ is a countable union of nonzero part of distributionally irregular vectors and $DIV=DIV\cap \ell^p(\mathbb{Z})=DIV\cap \ell^p(M_A)$, the set $DIV$ is connected by Corollary \ref{countable_union}.
\end{proof}

So far, it is clear that for some subspace $L^p(M)$, $DIV(T_{a,w})\cap L^p(M)$ is dense and connected in $L^p(M)$.
But, can we strengthen this result to the whole space $L^p(G)$?
That is, whether $DIV$ is connected or dense in $L^p(G)$.
Another question is if it is possible to show $\overline{DIV}=L^p(M)$ for some $M$.
Unfortunately, both answers are negative in general.
The following example will reveal that there exists a weighted translation
$T_{a,w}$ such that $DIV(T_{a,w})$ is neither connected nor dense in $L^p(G)$.
Moreover, this $DIV(T_{a,w})$ has exactly two connected components $DIV\cap L^p(M')$ and $DIV\cap L^p(M'')$ for some $M',M''$, and the closure of $DIV$ is $L^p(M')\cup L^p(M'')$.

\begin{example}
\label{counter_ex}
Define $B_{w'}:=T_{-1, w'}$ as the construction in the Example \ref{ex_Bw}.
Then both $B_{w'}$ and $B_{\frac{1}{w'}}$ are distributionally chaotic.

Let $G=\mathbb{Z}\times \mathbb{Z}_2$, $a=(-1,0)$ and
$$
w(n,m)=
\left\{\begin{matrix}
w'(n) &  \mbox{if}\ m=0;\\
\frac{1}{w'(n)} & \mbox{if}\ m=1.
\end{matrix}\right. 
$$
Then $T_{a,w}$ is distributionally chaotic and $DIV(T_{a,w})$ has exactly two connected components $DIV\cap L^p(\mathbb{Z}\times \{0\}),$ $DIV\cap L^p(\mathbb{Z}\times \{1\})$.
Moreover, 
$$\overline{DIV}=L^p(\mathbb{Z}\times \{0\})\cup L^p(\mathbb{Z}\times \{1\}).$$
\end{example}
\begin{proof}
Let $M':=\mathbb{Z}\times \{0\}$ and $M'':=\mathbb{Z}\times \{1\}$ and
let $\pi: \mathbb{Z}\times \mathbb{Z}_2 \to \mathbb{Z}$ be the natural projection.

Observe that
the action $B_{w'}$ on $\ell^p(\mathbb{Z})$ is isomorphic to the action $T_{a,w}|_{L^p(M')}$ on $L^p(M')$;
the action $B_{\frac{1}{w'}}$ on $\ell^p(\mathbb{Z})$ is isomorphic to the action $T_{a,w}|_{L^p(M'')}$ on $L^p(M'')$.
So $DIV\cap L^p(M')$ and $DIV\cap L^p(M'')$ are non-empty.

\noindent\textbf{Claim:}
$DIV\subseteq L^p(M')\cup L^p(M'')$.
On the other hand, if $DIV_{A,B}\neq\emptyset$,
then exactly one of the following occurs.
\begin{enumerate}
\item[{\rm(i)}]
$\overline{DIV_{A,B}} = L^p(M')$
\item[{\rm(ii)}]
$\overline{DIV_{A,B}} = L^p(M'')$.
\end{enumerate}

Let $y\in DIV$ with $A, B\subset \mathbb{N}$ as in the definition of $y$, i.e. $y\in DIV_{A,B}$.
Then for any $g\in\{y\neq 0\}$, $\chi_g \le \frac{1}{|y(g)|}|y|$. Hence $\chi_g \in X_{0,A}$.

Assume that $\{y\neq 0\}\cap M' \neq \emptyset$ and $\{y\neq 0\}\cap M'' \neq \emptyset$.
By Lemma \ref{convergent_and_divergent_for_any_subsequence}, for
$g'\in \{y\neq 0\}\cap M'$ and $g''\in \{y\neq 0\}\cap M''$, we have
$$\chi_{(0, 0)} = T_{a^{\pi(g')}}\chi_{g'} \in X_{0,A}$$
and
$$\chi_{(0, 1)} = T_{a^{\pi(g'')}}\chi_{g''} \in X_{0,A}.$$
This means that
$$\lim_{n\in A}\prod_{j=1}^{n}w'(j) = \lim_{n\in A}T_{a,w}^n\chi_{(0, 0)} = 0$$
and
$$\lim_{n\in A}(\prod_{j=1}^{n}w'(j))^{-1} = \lim_{n\in A}T_{a,w}^n\chi_{(0, 1)} = 0,$$
which is impossible.
So either $\{y\neq 0\}\cap M' = \emptyset$ or $\{y\neq 0\}\cap M'' = \emptyset$, that is
either $y\in L^p(M')$ or $y\in L^p(M'')$.
So
$$DIV\subseteq L^p(M')\cup L^p(M'').$$

Next, since $\overline{DIV_{A,B}} = L^p(M_A)\subset L^p(M')\cup L^p(M'')$,
then either
$M_A \subseteq M'$ or $M_A \subseteq M''$.

By the invertibility of $T_{a,w}$, Theorem \ref{m_a} forces either $M_A = M'$ or $M_A = M''$.

So the claim follows.

\noindent\textbf{Claim:}
$DIV\cap L^p(M')$ and $DIV\cap L^p(M'')$ are connected.

Since $DIV\cap L^p(M')$ is non-empty, let $y \in DIV\cap L^p(M')$.
Then
$M' = M_A$ for $A$ being as in the definition of $y$.
So $DIV\cap L^p(M')$ is connected by Theorem \ref{m_ab}.
Similar to the case for $DIV\cap L^p(M'')$.

Finally, the three facts
\begin{enumerate}
\item
Both $L^p(M')$ and $L^p(M'')$ are closed subspace,
\item
$L^p(M')\cap L^p(M'') = \{0\}$,
\item
$0\notin DIV$,
\end{enumerate}
implies that $DIV\cap L^p(M')$ and $DIV\cap L^p(M'')$ are both clopen in the subspace topology of $DIV$.
We conclude that
$DIV$ has exactly two connected components $DIV\cap L^p(M')$ and $DIV\cap L^p(M'')$.
\end{proof}

\noindent
\addcontentsline{toc}{section}{Acknowledgements}
\textbf{Acknowledgements.}
The authors would like to thank
Eugene Z. Xia,
Chung-Chuan Chen
and
Chun-Wei Lee
for their assistance.


\vspace{.1in}

\begin{thebibliography}{90}

\bibitem{ak17} E. Abakumov and Y. Kuznetsova, {\it Density of translates in weighted $L^p$ spaces on locally compact groups},
Monatsh. Math. \textbf{183} (2017) 397-413.

\bibitem{aa17} M. R. Azimi and I. Akbarbaglu, {\it Hypercyclicity of
weighted translations on Orlicz spaces}, Oper. Matrices, \textbf{12} (2018) 27-37.


\bibitem{bmbook} F. Bayart and \'E. Matheron, Dynamics of linear operators, Cambridge Tracts in Math. \textbf{179},
 Cambridge University Press, Cambridge, 2009.

\bibitem{bbmp11} T. Berm\'udez, A. Bonilla, F. Mart\'inez-Gim\'enez and A. Peris, {\it Li-Yorke and distributionally chaotic operators},
J. Math. Anal. Appl. \textbf{373} (2011) 83-93.

\bibitem{2013JFA} N. C. Bernardes Jr., A. Bonilla, V. M\"uler and A. Peris, {\it Distributional chaos for linear operators},
J. Funct. Anal. \textbf{265} (2013) 2143-2163.

\bibitem{chen172} C-C. Chen, {\it Disjoint hypercyclic weighted translations on groups}, Banach J. Math. Anal. \textbf{11} (2017) 459-476.

\bibitem{cc11} C-C. Chen and C-H. Chu, {\it Hypercyclic weighted translations on groups}, Proc. Amer. Math.
Soc. \textbf{139} (2011) 2839-2846.

\bibitem{cd18} C-C. Chen, K-Y. Chen, S. Oztop and S. M. Tabatabaie, {\it Chaotic translations on weighted Orlicz spaces},
Ann. Polon. Math. \textbf{122} (2019) 129-142.

\bibitem{ct18} C-C. Chen and S. M. Tabatabaie, {\it Chaotic operators on hypergroups}, Oper. Matrices, \textbf{12} (2018) 143-156.

\bibitem{kuchen17} K-Y. Chen, {\it On aperiodicity and hypercyclic
weighted translation operators}, J. Math. Anal. Appl. \textbf{462} (2018)
1669-1678.

\bibitem{marko16} J. A. Conejero, M. Kosti\' c, P. J. Miana and M. Murillo-Arcila,
{\it Distributionally chaotic families of operators on Fr\' echet spaces},
Comm. Pure Appl. Anal. \textbf{15} (2016) 1915-1939.

\bibitem{gpbook} K.-G. Grosse-Erdmann and A. Peris, Linear Chaos, Universitext, Springer, 2011.



\bibitem{mop09} F. Mart\'inez-Gim\'enez, P. Oprocha and A. Peris, {\it Distributional chaos for backward shifts},
J. Math. Anal. Appl. \textbf{351} (2009) 607-615.

\bibitem{mop13} F. Mart\'inez-Gim\'enez, P. Oprocha and A. Peris, {\it Distributional chaos for operators with full scrambled sets},
Math. Z. \textbf{274} (2013) 603-612.

\bibitem{o06} P. Oprocha, {\it A quantum harmonic oscillator and strong chaos},
J. Phys. A \textbf{39} (2006) 14559-14565.

\bibitem{ss94} B. Schweizer and J. Sm\'ital, {\it Measures of chaos and a spectral decomposition of dynamical systems on the interval},
Trans. Amer. Math. Soc. \textbf{344} (1994) 737-754.

\bibitem{zhang17} L. Zhang, H-Q. Lu, X-M. Fu and Z-H. Zhou, {\it Disjoint hypercyclic powers of weighted translations on groups},
Czechoslovak Math. J. \textbf{67} (2017) 839-853.

\bibitem{plig} R.A. Struble, {\it Metrics in locally compact groups},
Compos. Math., \textbf{28} (1974), 217-222.

\end{thebibliography}
\end{document}